\documentclass[12pt, a4paper]{article}

\usepackage[utf8x]{inputenc}
\usepackage[english]{babel}
\usepackage{amssymb}
\usepackage{amsfonts}
\usepackage{amsmath}
\usepackage{amsthm}
\usepackage{epsfig}
\usepackage{graphics}
\usepackage{dsfont}
\usepackage[usenames, dvipsnames]{xcolor}
\usepackage{verbatim}
\usepackage{latexsym}
\usepackage{setspace}

\usepackage{tikz}
\usepackage{enumerate}
\usepackage{pgfplots}
\usepackage{hyperref}
\usepackage{cancel}

\usepackage{caption} 
\usetikzlibrary{arrows,shapes,trees,patterns}

\theoremstyle{plain}
\newtheorem{theorem}{Theorem}[section]

\newtheorem{lemma}[theorem]{Lemma}
\newtheorem{proposition}[theorem]{Proposition}

\theoremstyle{definition}

\theoremstyle{definition}
\newtheorem{remark}[theorem]{Remark}

\allowdisplaybreaks[4]

\newcommand{\NN}{\mathbb{N}}

\newcommand{\RR}{\mathbb{R}}
\newcommand{\PP}{\mathbb{P}}
\newcommand{\CC}{\mathbb{C}}

\newcommand{\FF}{\mathbb{F}}

\newcommand{\EE}{\mathbb{E}}

\newcommand{\cL}{\mathcal{L}}

\newcommand{\dd}{\mathrm{d}}

\newcommand*\diff{\mathop{}\!\mathrm{d}}
\newcommand{\abs}[1]{\left\vert#1\right\vert}

\renewcommand{\tilde}{\widetilde}

\usepackage{marginnote}
\newcommand\normal{\color{black}}

\begin{document}
	\title{On moments of downward passage times for spectrally negative Lévy processes} 
	\author{Anita Behme\thanks{Technische Universit\"at
			Dresden, Institut f\"ur Mathematische Stochastik, Fakultät Mathematik, 01062 Dresden, Germany, \texttt{anita.behme@tu-dresden.de} and \texttt{philipp.strietzel@tu-dresden.de}, phone: +49-351-463-32425, fax:  +49-351-463-37251.}\; and Philipp Lukas Strietzel$^\ast$}
	\date{\today}
	\maketitle
	
	\vspace{-1cm}
	\begin{abstract} 
 	The existence of moments of first downward passage times of a spectrally negative Lévy process is governed by the general dynamics of the Lévy process, i.e. whether the Lévy process is drifting to $+\infty$, $-\infty$ or oscillates. Whenever the Lévy process drifts to $+\infty$, we prove that the $\kappa$-th moment of the first passage time (conditioned to be finite) exists if and only if the $(\kappa+1)$-th moment of the Lévy jump measure exists. This generalises a result shown earlier by Delbaen for Cramér-Lundberg risk processes \cite{Delbaen1990}. Whenever the Lévy process drifts to $-\infty$, we prove that all moments of the first passage time exist, while for an oscillating Lévy process we derive conditions for non-existence of the moments and in particular we show that no integer moments exist. 
	\end{abstract}
	
	(primary), 
	 91G05 
	(secondary) \\

	\normal	
	
	

	\section{Introduction}\label{S0}
	\setcounter{equation}{0}

Let $X=(X_t)_{t\geq 0}$ be a spectrally negative Lévy process, i.e. a L\'evy process that does not exhibit positive jumps, starting in zero. In this article we study moments of the \emph{first (downward) passage time of $-x$}, $x\geq 0$, of the process $X$, i.e. moments of
\begin{equation} \label{eq-firstpassage}	
	\tau_x^- := \inf \left\{ t>0: ~ X_t<-x \right\},
\end{equation}
conditioned on finiteness of this stopping time.

The first passage time $\tau_x^-$ - sometimes also referred to as \emph{exit time} - of (spectrally negative) L\'evy processes is a well-known object that has been studied by many authors, see \cite[Sec. 9.5]{DoneyBuch} for a general overview. However, most results are limited on giving a representation of the Laplace transform of the first passage time.\\ 
In case of a Brownian motion with drift $p\in\mathbb{R}$, due to the continuity of the paths, the first passage time $\tau_x^-$ coincides with the \emph{first hitting time} of $-x$, i.e. with  $\tau_x^{-,\ast} = \inf\{t> 0: X_t=-x\}$. In this special case, $\tau_x^-$ is known to have Laplace transform, cf. \cite[Eq. I.(9.1)]{rogerswilliams1},
\begin{equation}\label{eq_BrownianmotionLaplace}
	\EE[e^{-q \tau_x^-}]  = e^{-(\sqrt{p^2+2q}+p) x}, \quad x\geq 0, q>0, 
\end{equation}
and its distribution is given explicitly as, cf. \cite[Eq. I.(9.2)]{rogerswilliams1},
$$\PP(\tau_x^- \in \dd z)= \frac{x}{\sqrt{2\pi z^3}}e^{- \frac{(x+pz)^2}{2z}} \dd z,\quad x,z\geq 0, $$
where in both formulas we assumed the process to be standardized, i.e. such that $\sigma^2=1$. \\
For general spectrally negative Lévy processes the first hitting time and the first passage time can be related via the undershoot $U_x:=X_{\tau_x^-}+x\leq 0$ as shown in \cite{Doney1991}.\\ In particular, for spectrally negative $\alpha$-stable processes ($1<\alpha<2$) due to self-similarity this relation reads, cf. \cite{Simon2011},
\begin{equation} \label{eq-relation}
 \tau_x^{-,\ast} 
 \overset{d}= \tau_x^- + |U_x|^\alpha \hat{\tau}_1^{+,\ast}= x^\alpha (\tau_1^- + |U_1|^\alpha \hat{\tau}_1^{+,\ast})\end{equation}
where 
 $\hat{\tau}_x^{+,\ast}$ is an independent copy of the first upwards hitting time $\tau_x^{+,\ast}=\inf\{t>0: X_t=x\}$. 
The hitting time $\tau_x^{-,\ast}$ of a spectrally one-sided stable process has been studied e.g.  in \cite{Peskir2008, Simon2011, KuznetsovKyprianou2014}. In particular, in  \cite{Simon2011} fractional moments and a series representation of the density $\tau_x^{-,\ast}$ are provided. 

The first downward passage time $\tau_x^-$ has also been extensively studied in the field of actuarial mathematics where the spectrally negative L\'evy process $X$ is interpreted as \emph{risk process} and shifted to start in $x\geq 0$. Then, due to the space homogeneity of the Lévy process, $\tau_x^-$ coincides with the \emph{time of ruin}, i.e. the first time the process passes the value zero. The most prominent example for such a risk process is the classical \emph{Cram\'er-Lundberg model}, where $X$ is chosen to be a spectrally negative compound Poisson process, i.e.
\begin{equation}\label{eq-CLmodel}
	X_t=x+pt - \sum_{i=1}^{N_t} S_i, \quad t\geq 0.
\end{equation}
Hereby $x\geq 0$ is interpreted as \emph{initial capital}, $p>0$ denotes a constant \emph{premium rate}, the Poisson process $(N_t)_{t\geq 0}$ represents the \emph{claim counting process}, and the i.i.d. positive random variables $\{S_i, i\in\NN\}$ are the \emph{claim size variables} which are independent of $(N_t)_{t\geq 0}$.\\
For this model, under the profitability assumption $\EE[X_1]>0$, it is shown  in \cite{Delbaen1990} for all $\kappa>0$ that the $\kappa$-th moment of the ruin time exists, if and only if the $(\kappa+1)$-th moment of the claim size distribution exists. In this paper, we generalize this result to arbitrary spectrally negative L\'evy processes. Note that, while the proof given in \cite{Delbaen1990} relies on results on the speed of convergence of random walks, we use a completely different approach here via fractional differentiation of Laplace transforms. In particular, our approach allows us to relate the existence of $\EE[(\tau_x^-)^\kappa|\tau_x^-<\infty]$ with the existence of the $\kappa$-th moment of the subordinator $(\tau_x^+)_{x\geq 0}$ of upwards passage times $\tau_x^+=\inf\{t>0: X_t>x\}$ at a specific random time. As a by-product, we show that $(\tau_x^+)_{x\geq 0}$ is a special subordinator and identify its conjugate subordinator.

Before presenting and proving our main theorem on the existence of moments of the first passage time in Section \ref{S2}, we collect various preliminary results on (spectrally negative) L\'evy processes and fractional derivatives in Section \ref{S1}. 

	\section{Preliminaries}\label{S1}
	\setcounter{equation}{0}

Throughout this article let $X=(X_t)_{t\geq 0}$ be a Lévy process, i.e. a càdlàg stochastic process with independent and stationary increments, defined on a filtered  probability space $(\Omega,\mathcal{F}, \FF, \mathbb{P})$. 
It is well-known that the L\'evy process $X$ is fully characterized by its \textit{characteristic exponent} $\Psi$, which is defined via $e^{-t\Psi(\theta)} = \mathbb{E}[e^{i\theta X(t)}]$ and takes the form 
\begin{equation*}
\Psi(\theta) = 
 i a\theta + \frac{1}{2}\sigma^2 \theta^2 + \int_{\mathbb{R}}\left(1-e^{i\theta y} + i\theta y\mathds{1}_{\{\abs{y}<1\}}\right)\Pi^*(\diff y),  \quad \theta \in \RR,
\end{equation*} 
for constants $a\in\mathbb{R}$, $\sigma^2\geq 0$, and a measure $\Pi^*$ on $\mathbb{R}\backslash\{0\}$ satisfying $\int_{\mathbb{R}}(1~\wedge~y^2) \Pi^*(\diff y)<\infty$. The measure $\Pi^*$ is called the \emph{L\'evy measure} or \emph{jump distribution} of $X$, while $(\sigma^2,a,\Pi^*)$ is the \emph{characteristic triplet} of $X$.  \\ 
If $X$ has no upwards jumps, i.e. if $\Pi^*((0,\infty))=0$, then $X$ is called  \emph{spectrally negative}. In this case, it is handy to use the \emph{Laplace exponent} $\psi(\theta) :=  \frac{1}{t}\log\mathbb{E}[e^{\theta X_t}]$, $\theta \geq 0$, of $-X$ instead of the characteristic exponent, which then can be written in the form 
\begin{equation}\label{eq-Laplaceexp}
\psi(\theta) =  c \theta + \frac{1}{2}\sigma^2 \theta ^2 + \int_{(0,\infty)} \left(e^{-\theta y}-1+\theta  y \mathds{1}_{\{y<1\}} \right)\Pi(\diff y), 
\end{equation}
where $c=-a\in\mathbb{R}$, $\sigma^2 \geq 0$, and  $\Pi(\diff y) = \Pi^*(-\diff y)$ is the mirrored version of the jump measure which is therefore defined on $(0,\infty)$. 

The Laplace exponent $\psi$ admits some useful properties: \\Clearly $\psi(0)=0$, and $\lim_{\theta\to\infty}\psi(\theta) =\infty$. On $(0,\infty)$ the function $\psi$ is infinitely often differentiable and strictly convex. Lastly, as $\psi$ is nothing else than the cumulant generating function of $X_1$, it carries information on the moments of $X$. In particular, it is well-known, cf. \cite[Cor. 25.8]{sato2nd}, that for any $\kappa>0$ 
\begin{equation}\label{Lemma_equivalence_momentenBedingungen}
	\EE[|X_1|^\kappa]< \infty \quad \text{if and only if } \quad \int_{|y|\geq 1} |y|^\kappa~\Pi(\diff y)<\infty,\end{equation}
and for $\kappa=k\in \NN_0$ this in turn implies
\begin{equation} \label{eq-momentLaplace} |\partial^k \psi(0+)|:= |\psi^{(k)}(0+)|<\infty.\end{equation}
Note that throughout this article $\partial_q^k f(q,z)$ denotes the $k$-th derivative of a function $f$ with respect to $q$, while $\partial_q:= \partial_q^1$. In case of only one parameter, we will usually omit the subscript.

We will also use the Laplace exponent's right inverse which we always denote by
\begin{equation*}
\Phi(q) := \sup\{\theta\geq 0 : ~ \psi(\theta) = q\}, \quad q\geq 0. \end{equation*}
From the mentioned properties of $\psi$ it follows immediately that 
\begin{align*}
	\Phi(0)=0 & \quad \text{if and only if } \quad \psi'(0+)\geq 0, \\
	 \text{while} \quad  \Phi(0) >0 & \quad \text{if and only if } \quad \psi'(0+)<0.
\end{align*}
The function $q\mapsto \Phi(q)$ is strictly monotone increasing on $[0,\infty)$, infinitely often differentiable on $(0,\infty)$, and it is the well-defined inverse of $\psi(\theta)$ on the interval $[\Phi(0),\infty)$, i.e. 
\begin{equation*}
	\Phi(\psi(\theta)) = \theta \quad \text{ and } \quad  \psi(\Phi(q)) = q, \qquad \forall \theta \in [\Phi(0),\infty), ~ q\geq 0.
\end{equation*} 
Thus applying the chain rule on $q\mapsto q=\psi(\Phi(q))$ immediately yields
\begin{equation} \label{Lemma_derivative_inverse}
	\Phi'(q)=\partial_q \Phi(q) = \frac{1}{\psi'(\Phi(q))}, \quad q\geq 0,
\end{equation}
where the case $q=0$ is interpreted in the limiting sense $q\downarrow 0$.\\
Finally note that by definition  
\begin{equation}\label{eq-limiteta}
\lim_{q\downarrow 0} \frac{q}{\Phi(q)} = \begin{cases} \psi'(0+), & \text{if }\psi'(0+)\geq 0 , \\ 0, &\text{else.} \end{cases}
\end{equation}
For proofs of the stated properties and a more thorough discussion of Lévy processes in general we refer to \cite{Kyprianou2014} and \cite{sato2nd}.  

As announced in the introduction, we are interested in the first downward  passage time $\tau_x^-$ of $-x$, $x\geq 0$, as defined in \eqref{eq-firstpassage}, or, more precisely, in the first passage time given that the process passes through $-x$, i.e.
\begin{equation}\label{eq-tauconditioned}
	\left(\tau_x^-|\tau_x^-<\infty\right).
\end{equation}
Note that in the case that $ \psi'(0+) = \mathbb{E}[X_1]\in [-\infty, 0]$ we have $\tau_x^-=(\tau_x^-|\tau_x^-<\infty)$ as $X$ enters the negative half-line almost surely. In the case $\psi'(0+)>0$ the term \emph{passage time} will be typically used for the conditioned quantity \eqref{eq-tauconditioned}.\\
To avoid trivialities we exclude the case that $X$ is a pure drift, which implies a deterministic first passage time. Hence we always have $\PP(\tau_x^-<\infty)>0$. Moreover, we exclude the hitting level $x=0$ whenever $X_t$ is of unbounded variation, as in this case $\tau_0^- =0$ almost surely.

To study $\tau_x^-$ (or $(\tau_x^-|\tau_x^-<\infty)$) we will use the concept of scale functions. Recall that for any $q\geq 0$ the \emph{$q$-scale function} $W^{(q)}\colon\mathbb{R}\to[0,\infty)$  of the spectrally negative Lévy process $X$ is the unique function such that for $x\geq 0$ its Laplace transform satisfies 
\begin{equation*}
	\int_0^\infty e^{-\beta x}W^{(q)}(x) \diff x = \frac{1}{\psi(\beta)-q},
\end{equation*}
for all $\beta >\Phi(q)$. For $x<0$ we set $W^{(q)}(x)=0$. Furthermore the \emph{integrated $q$-scale function} $Z^{(q)}\colon\mathbb{R}\to[0,\infty)$ is given by 
\begin{equation}\label{eq_Zq}
	Z^{(q)}(x) := 1+q\int_0^x W^{(q)}(y)\diff y,
\end{equation}
and it fulfills, cf. \cite[Thm. 8.1]{Kyprianou2014},
\begin{equation} \label{eq_Kyprianou_Scale}
	\mathbb{E}\left[e^{-q\tau_x^-} \mathds{1}_{\{\tau_x^- <\infty\}} \right] = Z^{(q)}(x) - \frac{q}{\Phi(q)} \cdot W^{(q)}(x), \quad x\in\RR, q\geq 0.
\end{equation}
Taking the limit $q\downarrow 0$ this immediately implies 
\begin{equation} \label{eq_Kyprianou_Ruin}
	\mathbb{P}(\tau_x^-<\infty)  =  1- (0 \vee \psi'(0+)) \cdot W^{(0)}(x), \quad x\in \RR,
\end{equation}
where we use the standard notation $\vee$ to denote the maximum.\\
Observe that the functions $q\mapsto W^{(q)}(x)$ and $q\mapsto Z^{(q)}(x)$ may be extended analytically to $\mathbb{C}$, which means especially that they are infinitely often differentiable with bounded derivatives on  every compact set $K\subset\mathbb{C}$. This especially implies that limits of type $q\downarrow 0$ exist.
Again, we refer to \cite{Kyprianou2014} for missing proofs and further details. For detailed accounts on scale functions and their numerous applications, we also refer to \cite{Avram2020} and \cite{kuznetsov2011}. 

Lastly, let us recall that fractional moments of non-negative random variables can be computed via fractional differentiation of the corresponding Laplace transform. More precisely, define for any $\kappa \in(0,1)$ the \emph{Marchaud fractional derivative} of a function $f(z), z\geq 0$, 
\begin{equation} \label{eq_definition_marchaud_derivative}
\mathbf{D}^\kappa_z f(z) = \frac{\kappa}{\Gamma (1-\kappa)} \int_{z}^\infty \frac{f(z)-f(u)}{(u-z)^{\kappa+1}} \diff u,
\end{equation}
cf. \cite[Eq. (5.58)]{Samko1993}, while for $\kappa\geq 1$ with $n:=\lfloor \kappa \rfloor$ denoting the largest integer smaller or equal to $\kappa$
\[\mathbf{D}^\kappa_z f(z) = \partial_z^n \mathbf{D}^{\kappa-n}_z f(z).\]
Then, cf. \cite[Thm. 1]{Wolfe-frac}, for any non-negative random variable $T$ with Laplace transform $g(z)=\EE[e^{-zT}]$, $z\geq 0$, the $\kappa$-th absolute moment of $T$ exists, if and only if $\mathbf{D}_z^\kappa g(0)$ exists, in which case
\begin{equation} \label{eq-Wolfemoment} \mathbb{E}[T^\kappa] = \mathbf{D}_z^\kappa g(0).\end{equation}

This allows us to derive the following lemma.   

\begin{lemma} \label{lem-fracmoment1}
	For any $\kappa>0$, $x\geq 0$, the $\kappa$-th moment of the first downward passage time $\tau_x^-|\tau_x^-<\infty$ of a spectrally negative L\'evy process is given by
	\begin{equation} \label{eq_formula_to_work_with}
		\mathbb{E}\left[(\tau_x^-)^\kappa\big|\tau_x^- <\infty \right] = \frac{1}{ \mathbb{P}(\tau_x^- <\infty)}\cdot  \Big[\mathbf{D}^\kappa_q \Big(Z^{(q)}(x) - \frac{q}{\Phi(q)} \cdot W^{(q)}(x)\Big)\Big]_{q=0}
	\end{equation}
and it exists if and only if the right-hand side exists and is finite.
\end{lemma}
\begin{proof}
As 
$$\mathbb{E}\left[e^{-q\tau_x^-} \mathds{1}_{\{\tau_x^- <\infty\}} \right] = \mathbb{E}\left[e^{-q\tau_x^-}\big|\tau_x^- <\infty \right]  \cdot \mathbb{P}(\tau_x^- <\infty)$$
the claim follows immediately from \eqref{eq-Wolfemoment} and \eqref{eq_Kyprianou_Scale}.
\end{proof}

\section{Existence of moments} \label{S2}
\setcounter{equation}{0}

In \cite{Delbaen1990}, Delbaen showed in a classical Cram\'er-Lundberg model \eqref{eq-CLmodel} that is \emph{profitable}, i.e. with $\psi'(0+)>0$, that for any $\kappa>0$ the $\kappa$-th moment of the ruin time exists if and only if the $(\kappa+1)$-th moment of the claim sizes exists. Delbaen's proof relies on results on the speed of convergence of random walks. In this paper we use an alternative approach via fractional derivatives of Laplace transforms to prove an extension of the result in \cite{Delbaen1990} to any spectrally negative L\'evy process. Moreover, we additionally consider the \emph{non-profitable} settings of $\psi'(0+)\leq 0$.

Our main result in this section thus reads as follows. Note that although part (i) of Theorem \ref{thm-existence} seems to be known, we were unable to find a ready reference for this part and thus give a short proof below for the reader's convenience. 
\begin{theorem}\label{thm-existence}
	Let $(X_t)_{t\geq 0}$ be a spectrally negative L\'evy process with Laplace exponent $\psi$ as in \eqref{eq-Laplaceexp}, and let $\tau_x^-$ denote its first passage time of $-x$ for $x\geq 0$. 
	\begin{enumerate}
		\item If $\psi'(0+)<0$, then for any $x\geq 0$ there exists $q^\ast>0$ such that 
		\begin{equation*} 
			\mathbb{E}\big[e^{q^\ast \cdot \tau_x^-} \big] <\infty,
		\end{equation*} 
	   which implies for any $x\geq 0$ and $\kappa\geq 0$ 
		\begin{equation*} 
		\mathbb{E}\left[(\tau_x^-)^\kappa \right] <\infty.
		\end{equation*} 
		\item If $\psi'(0+) >0$, then for any $x \geq 0$ and $\kappa>0$ 
			\begin{equation*}
			\mathbb{E}\left[(\tau_x^-)^\kappa|\tau_x^-<\infty\right]<\infty \qquad \text{if and only if} \qquad \int_{[1,\infty)} y^{\kappa+1} \Pi(\diff y)<\infty.
		\end{equation*}
		\item Assume $\psi'(0+)=0$. 
		\begin{enumerate}[(a)]
		\item If there exists $\kappa^*\in (0,1]$ such that $\int_{[1,\infty)} y^{\kappa^*+1}\Pi(\diff y) =\infty$, then for any $x\geq 0$ and $\kappa\geq\kappa^*$
			\begin{equation} \label{eq_kappa-thmoment_infinite} 
			\mathbb{E}\left[(\tau_x^-)^\kappa\right] =\infty.
		\end{equation}
		\item If $\psi''(0+)<\infty$, then \eqref{eq_kappa-thmoment_infinite} holds for any $x\geq0$ and $\kappa>\tfrac{1}{2}$.
		\end{enumerate}
		In particular, \eqref{eq_kappa-thmoment_infinite} holds for any $x \geq 0$ and $\kappa\geq 1$.
	\end{enumerate}
\end{theorem}

\begin{remark}
	Note that a priori the above theorem needs no restrictions concerning possible choices of the location parameter $c\in\RR$ of $(X_t)_{t\geq 0}$. However, as by \cite[Ex. 25.12]{sato2nd},
	\begin{equation} \label{eq-relationcpsi} c - \int_{[1,\infty)} y \Pi(\diff y)= \EE[X_1]= \psi'(0+), \end{equation}
	in cases (ii) and (iii) the assumption $\psi'(0+)\geq 0$ implies that actually $$c\geq  \int_{[1,\infty)} y \Pi(\diff y)\geq 0.$$
	In particular $c<0$ is a valid choice only in case (i).
\end{remark}

\begin{remark}
	At first glance, Theorem \ref{thm-existence} (iii)  suggests that for an oscillating process $(X_t)_{t\geq 0}$ no fractional moments of the first passage time of zero exist. This, however, is not true in general and we provide two counterexamples:
	\begin{enumerate}
		\item Consider a (standardized) Brownian motion without drift for which by \eqref{eq_BrownianmotionLaplace}
		 \begin{equation*}
			\mathbb{E}\left[e^{-q\cdot \tau_x^-}\right] = e^{-\sqrt{2q}\cdot x}, \qquad x\geq 0.
		\end{equation*}
	Then $\PP(\tau_x^-<\infty)=1$ and from \eqref{eq_definition_marchaud_derivative} and \eqref{eq-Wolfemoment} we obtain for any $\kappa\in(0,1)$ that 
	\begin{align}
			\mathbb{E}[(\tau_x^-)^\kappa] &= \left[ \mathbf{D}^\kappa_q e^{-\sqrt{2q}\cdot x}\right]_{q=0}   
			= \frac{\kappa}{\Gamma(1-\kappa)} \int_0^\infty \frac{1 - e^{-\sqrt{2u}\cdot x} }{u^{\kappa +1}} \diff u,
		\end{align}
			which is finite if and only if $\kappa \in(0,\tfrac{1}{2})$. In particular, in this case \eqref{eq_kappa-thmoment_infinite} holds for any $\kappa\geq \tfrac{1}{2}$ which shows that Theorem \ref{thm-existence} (iii) (b) is near to being sharp.
		\item Consider a spectrally negative, $\alpha$-stable L\'evy process $(X_t)_{t\geq 0}$,  with index $\alpha\in(1,2)$, such that the Laplace exponent of $-X$ is given by $\psi(\theta) = \theta^\alpha$ and in particular $\psi'(0)=0$. For such a process it has been shown in \cite[Prop. 4 and subsequent Rem.]{Simon2011} that the first passage time $\tau_x^{-}$
		admits finite fractional moments, namely
		\[ \mathbb{E}[(\tau_x^{-})^\kappa]<\infty \quad \text{if and only if}\quad \kappa \in (-1,1-1/\alpha).\]
	\end{enumerate} 
	The identification of the threshold $\kappa^*\leq 1$  such that $\mathbb{E}[(\tau_x^-)^\kappa]<\infty$, $\kappa<\kappa^*$ and $\mathbb{E}(\tau_x^-)^\kappa] = \infty $ for $\kappa\geq \kappa^*$ for a general oscillating and spectrally negative Lévy process seems difficult: The chosen approach for our proof of Theorem \ref{thm-existence} below only yields the sufficient condition for \eqref{eq_kappa-thmoment_infinite} as stated in Theorem \ref{thm-existence}(iii)(a). Moreover, the above example of a Brownian motion clearly shows that the threshold $\kappa^\ast$ can not be solely depending on the Lévy measure. We therefore leave this question open for future research.
\end{remark}

\begin{proof}[Proof of Theorem \ref{thm-existence} (i)] 
Recall that if $\psi'(0+)<0$, then $\Phi(0)>0$ and by convexity of $\psi$ we obtain $\psi'(\Phi(0))>0$. Furthermore, the Laplace exponent $\psi$ extended to $\mathbb{C}$ is analytic on $\{x\in\mathbb{C}: ~\operatorname{Re}(z)>0\}$ and hence it is analytic in a neighborhood of $\Phi(0)$.  Consequently, the inverse function theorem of complex analysis, cf. \cite[Thm. 10.30]{rudincomplex} implies that also the inverse of $\psi(\theta)$, i.e., $\Phi(q)$, is analytic in a neighborhood of zero. By \eqref{eq_Kyprianou_Scale} this yields that $\mathbb{E}[e^{-q\tau_x^-}]$ is analytic in a neighborhood of zero as well, and thus $\mathbb{E}[e^{q^*\tau_x^-}]<\infty$ for some $q^*>0$ as claimed. 
\end{proof}

To prove the second and third part of Theorem \ref{thm-existence} we start with a simple lemma that reduces the problem of existence of moments of the first passage time to finiteness of (fractional) derivatives of a certain function in zero.
\begin{lemma} \label{Lemma_k-thmoment_eta}
	Set $\eta(q):=\frac{q}{\Phi(q)}$, $q>0$. Then for any $x\geq0$  and $\kappa>0$ 
	\begin{equation} \label{eq_MomentTau_FractDerivEta}
		\mathbb{E}\left[(\tau_x^-)^\kappa|\tau_x^-<\infty\right] <\infty \qquad \text{if and only if} \qquad  \lim_{q\downarrow 0} \abs{\mathbf{D}_q^\kappa \eta(q)} <\infty. 
	\end{equation} 
\end{lemma}
\begin{proof}
It follows immediately from Lemma \ref{lem-fracmoment1} that $\mathbb{E}\left[(\tau_x^-)^\kappa|\tau_x^-<\infty\right] <\infty$ if and only if $\lim_{q\downarrow 0}\mathbf{D}^\kappa_q \left(Z^{(q)}(x) - \eta(q) \cdot W^{(q)}(x)\right)<\infty$. However, $q\mapsto W^{(q)}(x)$ and $q\mapsto Z^{(q)}(x)$ are infinitely often differentiable with bounded derivatives on $[0,\infty)$. Hence, linearity of the (fractional) derivative reduces the problem to the characterisation of finiteness of $\lim_{q\downarrow 0}\mathbf{D}^\kappa_q \left( \eta(q) \cdot W^{(q)}(x)\right)$. \\
Observe that the definition of the Marchaud derivative is equivalent to the Liouville derivative for sufficiently good functions, see \cite[Remark 5.3]{Samko1993} for details. We may therefore apply the product rule for fractional Liouville derivatives, cf. \cite[p. 206]{Uchaikin2013}, to $\eta(q) \cdot W^{(q)}(x)$. Recalling again that $q\mapsto W^{(q)}(x)$ is infinitely often differentiable with bounded derivatives on every compact $K\subset \CC$, and that $W^{(q)}(x)>0$ for any $x>0$, we conclude that $\lim_{q\downarrow 0}\mathbf{D}^\kappa_q \left(\eta(q) \cdot W^{(q)}(x)\right)<\infty$ if and only if $\lim_{q\downarrow 0}\mathbf{D}^\kappa_q \eta(q) <\infty$ as claimed. \\
If $x=0$ note that $W^{(q)}(0)>0$ if and only if $(X_t)_{t\geq 0}$ is of bounded variation, cf. \cite[Eq. (25)]{Avram2020}, and in this case the above argumentation yields the result. 
\end{proof}

The remainder of the proof of Theorem \ref{thm-existence} relies on the interpretation of $\eta(q)$ as Laplace exponent of a certain killed subordinator as shown in the next proposition. Recall that a \emph{subordinator} $(Y_t)_{t \geq 0}$ is a L\'evy process with non-decreasing paths whose Laplace exponent 
$\varphi(\theta) = - \frac{1}{t} \log \EE[e^{-\theta Y_t} ]$ is of the form
\begin{equation} \label{eq_LaplaceExponent_Subordinator}
\varphi(\theta)= \tilde{c}\cdot\theta+\int_0^\infty (1-e^{-\theta y}) \tilde{\Pi}(\diff y),
\end{equation}
for $\theta\geq 0$, a drift $\tilde{c}\geq 0$ and a measure $\tilde{\Pi}$ such that $\int_{(0,\infty)} (1\wedge y) \tilde{\Pi}(\diff y)<\infty$. 
A \emph{killed} subordinator $(Y_t)_{t\geq 0}$ is defined via 
\begin{equation*}
	Y_t = \begin{cases} \tilde{Y}_t, & \text{if } t<\mathbf{e}_{\beta}, \\ 
		\zeta, & \text{if }t\geq \mathbf{e}_{\beta},
	\end{cases}
\end{equation*}
where $(\tilde{Y}_t)_{t\geq 0}$ is a subordinator, $\mathbf{e}_{\beta}$ is an independent $\operatorname{Exp}(\beta)$-distributed time, $\beta>0$, and $\zeta$ denotes some \emph{cemetery state}. As usual, we interpret $\beta=0$ as $\mathbf{e}_\beta = \infty$ corresponding to no killing. The Laplace exponent $\varphi_Y$ of a killed subordinator is given by 
\begin{equation} \label{eq_LaplaceExponent_KilledSubordinator}
	\varphi_Y(\theta) = - \log \mathbb{E}\left[e^{-\theta Y_1}\right] = - \log \mathbb{E}\left[e^{-\theta \tilde{Y}_1}\cdot \mathds{1}_{\{1<\mathbf{e}_{\beta}\}}\right] = \beta + \varphi_{\tilde{Y}}(\theta),
\end{equation}
for the Laplace exponent $\varphi_{\tilde{Y}}(\theta)$ of $(\tilde{Y}_t)_{t\geq 0}$.

Further, for any $x\geq 0$, let  $\tau_x^+ := \inf\{ t>0: ~ X_t >x\}$ be the first upwards passage time of $x$, i.e. the first time that $X_t$ is above $x$. It is well-known, cf. \cite[Thm. 3.12]{Kyprianou2014}, that for all $q\geq 0$
\begin{equation} \label{eq_Laplace_taux+}
	\mathbb{E}\left[e^{-q\cdot \tau_x^+}\cdot\mathds{1}_{\{\tau_x^+<\infty\}}\right] = e^{-\Phi(q)x}, \quad x\geq 0.
\end{equation}
If furthermore $\mathbb{E}[X_1]=\psi'(0+)\geq 0$, then $(\tau_x^+)_{x\ge0 }$ is a subordinator with Laplace exponent $\Phi(q)$, cf. \cite[Cor. 3.14]{Kyprianou2014}.

\begin{proposition}\label{Prop-etaalsLE}
Assume $\psi'(0+)\geq 0$ and define a killed subordinator $(Y_t)_{t\geq 0}$, independent of $(\tau_x^+)_{x\geq 0}$, through its Laplace exponent $\varphi(\theta)= -t^{-1} \log \mathbb{E}[e^{-\theta Y_t}]$ by setting
\begin{equation*}
	\varphi(\theta) :=  \psi'(0+)  + \frac{\sigma^2}{2} \theta  + \int_0^\infty  \big(1  - e^{-\theta y}\big)  \Pi((y,\infty)) \diff y, \quad \theta > 0.
\end{equation*} 
 Then $(\tau_{Y_t}^+)_{t\geq 0}$ is a killed subordinator with Laplace exponent 
	\begin{equation} \label{eq_laplacetaux+}
	-\frac{1}{t}\log \mathbb{E}\left[e^{-q\cdot \tau_{Y_t}^+} \right]  = \eta(q), \quad q\geq 0.
	\end{equation}
\end{proposition}
\begin{proof}
An application of \cite[Thm. 1]{Kyprianou2008} on the subordinator $(Y_t)_{t\geq 0}$ implies that there exists a spectrally negative Lévy process - the so-called \emph{parent process} - with Laplace exponent $\theta \cdot \varphi(\theta)$ and whose characteristic triplet coincides with the one of $(X_t)_{t\geq 0}$. Therefore
\begin{equation*}
\varphi(\theta) = \frac{\psi(\theta)}{\theta}.
\end{equation*}
Equation \eqref{eq_laplacetaux+} is now a direct consequence of \cite[Thm. 30.1]{sato2nd} and the fact that 
\begin{equation*}
\varphi(\Phi(q))= \frac{\psi(\Phi(q))}{\Phi(q)} = \frac{q}{\Phi(q)}=\eta(q). \qedhere
\end{equation*}
\end{proof}

\begin{remark} Note that the above proposition implies that - as long as $\psi'(0+)\geq 0$ - the subordinator $(\tau^+_x)_{x\geq 0}$ is a \emph{special subordinator} since its \emph{conjugate} Laplace exponent $\frac{q}{\Phi(q)}=\eta(q)$ is shown to be the Laplace exponent of a (killed) subordinator. See e.g. \cite[Chapter 5.6]{Kyprianou2014} or \cite[Chapter 11]{rene-book} for general information on special subordinators and their Laplace exponents that are also known as \emph{special Bernstein functions}.
	\end{remark}

Combining Lemma \ref{Lemma_k-thmoment_eta},  Proposition \ref{Prop-etaalsLE}, and Equation \eqref{eq-Wolfemoment} it is an immediate consequence that, assuming $\psi'(0+)\geq 0$, for all $\kappa>0$ and $x\geq 0$
\begin{equation} \label{eq_tau-_tau+}
	\mathbb{E}\left[(\tau_x^-)^\kappa|\tau_x^-<\infty\right] <\infty \qquad \text{if and only if} \qquad  \EE[(\tau_{Y_1}^+)^\kappa] <\infty.
\end{equation}

In order to find suitable conditions for the right-hand side of \eqref{eq_tau-_tau+}, we next prove a general statement concerning the existence of moments of a subordinated subordinator.

\begin{proposition}\label{lem-momentsubordination} 
	Let $(Z_t)_{t\geq 0}$ be a non-zero subordinator, and let $(Y_t)_{t\geq 0}$ be a (possibly killed) non-zero subordinator, independent of $(Z_t)_{t\geq 0}$. If $\mathbb{E}[Z_1]<\infty$, then for all $\kappa>0$ 
	$$\EE[Z_{Y_1}^\kappa] <\infty  \qquad \text{if and only if} \qquad \Big[ \EE[Z_1^\kappa] <\infty  \text{ and } \EE[Y_1^{\kappa}] <\infty \Big].$$
	If $\mathbb{E}[Z_1]=\infty$ and $\kappa\in (0,1)$, then $\EE[Z_{Y_1}^\kappa] <\infty$ implies $\EE[Z_1^\kappa] <\infty$ and $\EE[Y_1^{\kappa}] <\infty$.
\end{proposition}

To prove this proposition, we need the following lemma. 

\begin{lemma} \label{Lemma_taux+_polynomial} 
	Let $(Z_t)_{t\geq 0}$ be a non-zero subordinator such that $\mathbb{E}[Z_1]<\infty$. If $\EE[Z_1^\kappa]<\infty$ for some $\kappa>0$, then $\EE[Z_t^\kappa]<\infty$ for all $t\geq 0$ and the mapping $t\mapsto \EE[Z_t^\kappa]$,  $t\geq 1$, is of polynomial order $\kappa$.
\end{lemma}
\begin{proof}
First note that by Hölder's inequality for all $n\in \NN$, $a_1,...,a_n\geq 0$ and $r\geq 1$ 
	\begin{equation} \label{eq_polyhelper}
		(a_1+...+a_n)^r \leq n^{r-1}\cdot \left( a_1^r + ...+ a_n^r \right),
	\end{equation} 
while for $r\leq 1$ inequality \eqref{eq_polyhelper} holds with `` $\geq$'' instead of `` $\leq$''. \\
Let $\varphi$ be the Laplace exponent of the subordinator $(Z_t)_{t\geq 0}$. By \cite[Cor. 25.8]{sato2nd} finiteness of $\EE[Z_1^\kappa]$ for some $\kappa>0$ implies finiteness of $\EE[Z_t^\kappa]$ for all $t\geq 0$.\\
	By our assumptions, $\mathbb{E}[Z_1] = \varphi'(0+)\in(0,\infty)$ and it follows that, cf. \cite[Ex. 25.12]{sato2nd},
	\begin{equation*}
		\mathbb{E}[Z_t] = 
		t \cdot \varphi'(0+) =t\cdot \mathbb{E}[Z_1].
	\end{equation*}
	Let now $\kappa \geq 1$. Then by Jensen's inequality we conclude that
	\begin{equation*}
		\mathbb{E}[Z_t^\kappa] \geq  \mathbb{E}[Z_t]^\kappa  = t^\kappa \cdot \mathbb{E}[Z_1]^\kappa,
	\end{equation*} which yields a lower bound of degree $\kappa$. 
	In order to show an upper bound 
	set $n:=\lceil t\rceil$ such that $t/n =:c_n\in [\tfrac12 ,1]$ and let $\xi_i$ be i.i.d. copies of $Z_1$. Then, due to the infinite divisibility and monotonicity of $Z$, it holds that
	\begin{equation}\label{eq_proof_polynomial}
		\begin{aligned} 
			\mathbb{E}[Z_t^\kappa] &\leq  \mathbb{E}[Z_n^\kappa] = \mathbb{E}\Big[ \Big( \sum_{i=1}^{n} \xi_i \Big)^\kappa   \Big]  \\ 
			&\leq \mathbb{E}\Big[ n^{\kappa-1} \cdot \sum_{i=1}^{n} \xi_i^\kappa   \Big] 
			= n^\kappa \cdot \mathbb{E}[\xi_1^\kappa] = t^\kappa \cdot c_n^{-\kappa} \cdot \mathbb{E}[Z_1^\kappa] \leq t^\kappa \cdot 2^{\kappa} \cdot \mathbb{E}[Z_1^\kappa],
		\end{aligned}
	\end{equation}
	where we used \eqref{eq_polyhelper} for the second inequality.  \\ 
	To prove the statement for $\kappa\in(0,1)$ note that $({}\cdot{})^\kappa$ is concave. Hence, Jensen's inequality yields an upper bound in this case. The lower bound follows analogously to \eqref{eq_proof_polynomial}, setting $n:=\lfloor t\rfloor$, applying the variant of  \eqref{eq_polyhelper} for $r\leq 1$.
	\end{proof}

\begin{proof}[Proof of Proposition \ref{lem-momentsubordination}]
	We write $\nu_Z$, $b_Z$ for the L\'evy measure and drift of $Z$, respectively, and likewise $\nu_Y$, $b_Y$ for L\'evy measure  and drift of $Y$. Then as $(Z_{Y_t})_{t\geq 0}$ is a (killed) subordinator with L\'evy measure $\nu$, say, $\EE[Z_{Y_1}^\kappa] <\infty$ is equivalent to, cf. \cite[Cor. 25.8]{sato2nd},
	\begin{equation}\label{eq-kappamomenttausub}
		\int_{[1,\infty)} z^\kappa \nu(\diff z)<\infty.
	\end{equation}
	where the L\'evy measure $\nu$ of the subordinated process is given by, cf. \cite[Thm. 30.1]{sato2nd},
	$$\nu(B)= b_Y \nu_Z(B) + \int_{(0,\infty)} \mu^s(B) \nu_Y (\diff s),$$
	for any Borel set $B$ in $(0,\infty)$, where $\mu=\cL(Z_1)$ denotes the  distribution of $Z_1$. Thus
	\begin{align}
		\int_{[1,\infty)} z^\kappa \nu(\diff z) &= b_Y \int_{[1,\infty)} z^\kappa \nu_Z(\diff z) + \int_{[1,\infty)} z^\kappa \diff \left(\int_{(0,\infty)}  \mu^s (z) \nu_Y (\diff s)\right)  \nonumber \\
		&= b_Y \int_{[1,\infty)} z^\kappa \nu_Z(\diff z) + \int_{(0,\infty)}  \int_{[1,\infty)} z^{\kappa} \mu^s (\diff z) \nu_Y (\diff s) \label{eq_subordjumps}
	\end{align}
	where all terms are non-negative and hence the appearing sum is finite if and only if both summands are finite. From \cite[Cor. 25.8]{sato2nd} we know that $\int_{[1,\infty)} z^\kappa \nu_Z(\diff z)<\infty$ if and only if $\EE[Z_1^\kappa]<\infty$ if and only if $\EE[Z_s^\kappa]<\infty$ for all $s\geq 0$. Thus assume $\int_{[1,\infty)} z^\kappa \nu_Z(\diff z)<\infty$ from now on, which implies $\int_{[1,\infty)} z^{\kappa} \mu^s(\diff z) = \EE[\mathds{1}_{\{Z_s \geq 1\}} Z_s^\kappa] <\infty$.  
	Furthermore 
	\begin{equation} \label{eq_proof_PropSubord-Subordinator_1}
		\begin{aligned}
			&\int_{(0,\infty)}  \int_{[1,\infty)} z^{\kappa} \mu^s (\diff z) \nu_Y (\diff s) 
			= \int_{(0,\infty)} \mathbb{E}[\mathds{1}_{\{Z_s\geq 1\}} Z_s^\kappa ] \nu_Y (\diff s) \\ 
			&=   \int_{(0,1)} \mathbb{E}[\mathds{1}_{\{Z_s\geq 1\}} Z_s^\kappa ] \nu_Y (\diff s) + \int_{[1,\infty)} \mathbb{E}[ Z_s^\kappa ] \nu_Y (\diff s) -\int_{[1,\infty)} \underbrace{\mathbb{E}[\mathds{1}_{\{Z_s<1 \}} Z_s^\kappa ]}_{\in [0,1)} \nu_Y (\diff s),
		\end{aligned}
	\end{equation}
	where the left-hand side of the equation is finite if and only if the right-hand side is finite. Hereby, the last integral, as well as the sum of all three, is non-negative.\\ 
Consider the first integral. It holds that
	\begin{align*}
		\mathbb{E}[\mathds{1}_{\{Z_s\geq 1\}} Z_s^\kappa] = \mathbb{P}(Z_s\geq 1) \cdot \mathbb{E}\left[Z_s^\kappa \big| Z_s\geq 1\right],
	\end{align*}
	where  $\mathbb{E}[Z_s^\kappa | Z_s\geq 1]=:C_1(s)<\infty$, $s\in[0,\infty),$  since we assumed $\mathbb{E}[Z_s^\kappa]<\infty$. 
	From \cite[Lemma 30.3]{sato2nd} it follows that $\mathbb{P}(Z_s\geq 1)\leq C_2 s$ for some $C_2\in (0,\infty)$. Thus, setting $C_1:=\sup_{s\in(0,1)}C_1(s)<\infty$, 
	\begin{align*}
		\int_{(0,1)} \mathbb{E}[\mathds{1}_{\{Z_s\geq 1\}} Z_s^\kappa ] \nu_Y (\diff s)\leq C_1C_2 \int_{(0,1)} s \nu_Y (\diff s), 
	\end{align*} is finite because $Y$ is a subordinator which implies $\int_{(0,\infty)} (1\wedge y)\nu_Y(\diff y)<\infty$. \\
	For the second integral note that by Lemma \ref{Lemma_taux+_polynomial} the mapping $s\mapsto \EE[Z_s^\kappa]$ is of polynomial order $\kappa$ for all $s\geq 1$. 
	Thus it follows that the second summand in \eqref{eq_proof_PropSubord-Subordinator_1} and hence also the second summand
	in \eqref{eq_subordjumps} is finite if and only if 
	$\int_{[1,\infty)}  s^\kappa \nu_Y(\diff s)<\infty$ and $\int_{[1,\infty)} z^\kappa \nu_Z(\diff z)<\infty$. This finishes the proof of the claimed equivalence.\\
	In the case $\EE[Z_1]=\infty$ and $\kappa \in(0,1)$ we can not apply Lemma \ref{Lemma_taux+_polynomial} to find a necessary and sufficient condition for finiteness of the second summand in \eqref{eq_proof_PropSubord-Subordinator_1}. However, an inspection of the proof of Lemma \ref{Lemma_taux+_polynomial} shows that even in this case the mapping  $s\mapsto \EE[Z_s^\kappa]$ can be bounded from below by a function of polynomial order $\kappa$ for all $s\geq 1$. Thus finiteness of the second summand in \eqref{eq_proof_PropSubord-Subordinator_1} still implies $\int_{[1,\infty)} z^\kappa \nu_Z(\diff z)<\infty$, and  finiteness of all summands in \eqref{eq_subordjumps} implies $\int_{[1,\infty)}  s^\kappa \nu_Y(\diff s)<\infty$ and $\int_{[1,\infty)} z^\kappa \nu_Z(\diff z)<\infty$ as claimed.
\end{proof}

Let us now concentrate on the case $\psi'(0+)>0$ treated in Theorem \ref{thm-existence}(ii), where in the light of \eqref{eq_tau-_tau+} it remains to be proven that $\EE[(\tau_{Y_1}^+)^\kappa]<\infty$ is equivalent to $\int_{[1,\infty)} y^{\kappa+1} \Pi (\diff y)<\infty$. To show this we need the following 
useful \normal connection between the existence of integer moments of  $\tau^+_1$ and $X_1$.

\begin{lemma} \label{Lemma_psiphin}
	Assume that $\psi'(0+)> 0$. Then for all $k\in\NN_0$ 
	\begin{equation} \label{eq_proof_conj_psiphin}
		\lim_{q\downarrow 0 }\abs{ \Phi^{(k)}(q)} < \infty \qquad \text{if and only if} \qquad \lim_{q\downarrow 0}\abs{\psi^{(k)}(q)}<\infty.
	\end{equation}
\end{lemma}
\begin{proof}
	We prove the statement by induction.	
	Clearly, for $k=0$ there is nothing to show. For $k=1$ it follows from the assumption $\psi'(0+)> 0$ and the fact that $(X_t)_{t\geq 0}$ is spectrally negative, that $\psi'(0+)\in(0,\infty)$. By \eqref{Lemma_derivative_inverse} we thus conclude that $\Phi'(0+)=1/\psi'(0+)\in(0,\infty)$ and the equivalence is trivially fulfilled.  Further, for $k=2$ we compute via \eqref{Lemma_derivative_inverse} 
	\begin{align*}
		\Phi''(q)&= \partial_q\left( \frac{1}{\psi'(\Phi(q))}\right) = - \frac{\psi''(\Phi(q))}{\psi'(\Phi(q))^3}, \quad q>0,
	\end{align*} 
	such that 
	$$\Phi''(0+) = - \frac{\psi''(0+)}{\psi'(0+)^3}$$
	which proves the claim for $k=2$. \\
	Assume now that \eqref{eq_proof_conj_psiphin} holds for all $\ell=1,...,n-1$. If there exists $\ell'\in\{1,...,n-1\}$ such that both sides of \eqref{eq_proof_conj_psiphin} are infinite, then for all $\ell\in\{\ell',...,n-1\}$ both terms are infinite as well. Therefore we assume that both sides are finite for all $\ell=1,...,n-1$. \\
	By definition of $\Phi$ we have $\psi(\Phi(q))=q$ for all $q\geq 0$ and hence $\partial_q^n \psi(\Phi(q))=0$ for all $n\geq 2$. Using Faà di Bruno's formula, cf. \cite[Equation (2.2)]{Johnson2002}, for $n\geq 2$ we therefore conclude that
	\begin{align*}
		0 &= \sum_{k=1}^n \psi^{(k)}(\Phi(q)) \cdot B_{n,k}(\Phi'(q),...,\Phi^{(n-k+1)}(q)),
	\end{align*}
	where the functions $B_{n,k}$ denote the partial Bell polynomials. Thus we get 
	\begin{align*}
		\Phi^{(n)}(q) = B_{n,1}(\Phi^{(n)}(q)) &= \frac{-1}{\psi'(\Phi(q)}\cdot \sum_{k=2}^n \psi^{(k)}(\Phi(q)) \cdot B_{n,k}\left(\Phi'(q),...,\Phi^{(n-k+1)}(q)\right) \\
		&= \frac{-1}{\psi'(\Phi(q))}\cdot \sum_{j=1}^{n-1} \psi^{(n+1-j)}(\Phi(q)) \cdot B_{n,n+1-j}\left(\Phi'(q),...,\Phi^{(j)}(q)\right),
	\end{align*}
	where the left-hand side is finite if and only if the right-hand side is finite. However, the right-hand side is finite in the limit $q\downarrow 0$ if and only if 
	\begin{align*}
		\lim_{q\downarrow 0} \abs{\frac{1}{\psi(\Phi(q))} \cdot \psi^{(n)}(\Phi(q)) \cdot B_{n,n}(\Phi'(q))} & = 		\lim_{q\downarrow 0} \abs{\frac{1}{\psi(\Phi(q))} \cdot \psi^{(n)}(\Phi(q)) \cdot \Phi'(q)^n} \\ & = \abs{\frac{\psi^{(n)}(0+)}{\psi'(0+)^{n+1}}}<\infty,
	\end{align*}
	since all other summands are finite in the limit $q\downarrow 0$ by assumption. 
\end{proof}

 We are now in the position to present the proof of part (ii) of Theorem \ref{thm-existence}.

\begin{proof}[Proof of Theorem \ref{thm-existence}(ii)]
	Assume $\psi'(0+)> 0$. 
	By using Lemma \ref{Lemma_k-thmoment_eta},  Proposition \ref{Prop-etaalsLE} and \eqref{eq-Wolfemoment} we see immediately that for all $\kappa>0$
	\begin{equation*}
		\mathbb{E}\left[(\tau_x^-)^\kappa|\tau_x^-<\infty\right] <\infty \qquad \text{if and only if} \qquad  \EE[(\tau_{Y_1}^+)^\kappa] <\infty. 
	\end{equation*}
Further applying Proposition \ref{lem-momentsubordination} it follows that 
	\begin{equation} \label{eq-momentssubordinator} \EE[(\tau_{Y_1}^+)^\kappa] <\infty  \qquad \text{if and only if} \qquad \left[ \EE[(\tau_1^+)^\kappa] <\infty  \text{ and } \EE[Y_1^\kappa ]<\infty \right],\end{equation} 
since $\mathbb{E}[\tau_1^+]= \Phi'(0+)=1/\psi'(0+)<\infty$ as noted in the proof of Lemma \ref{Lemma_psiphin}.  Furthermore, $\EE[Y_1^\kappa ]<\infty $ is equivalent to finiteness of
	\begin{align}\label{eq-proof-partialintegration}
		\int_{[1,\infty)} y^{\kappa} \Pi((y,\infty)) \diff y &= \frac{1}{\kappa+1} \int_{[1,\infty)} y^{\kappa+1} \Pi(\diff x) -  \frac{1}{\kappa+1}\Pi((1,\infty))
	\end{align}
by partial integration. Thus 
$$\EE[Y_1^\kappa ]<\infty \qquad \text{if and only if} \qquad  \EE[|X_1|^{\kappa+1} ]<\infty$$
such that $\EE[|X_1|^{\kappa+1} ]<\infty$ is shown to be a necessary condition for $\mathbb{E}\left[(\tau_x^-)^\kappa|\tau_x^-<\infty\right] <\infty$. However, $\mathbb{E}[|X_1|^{\kappa+1}]< \infty$ is a sufficient condition as well, since it implies $\mathbb{E}[|X_1|^{k}]<\infty$ for $k=\lfloor \kappa +1 \rfloor\geq 1$. This in turn implies  $\mathbb{E}[(\tau^+_1)^{k}]<\infty$ by \eqref{eq-momentLaplace} and Lemma \ref{Lemma_psiphin}, which then yields $\mathbb{E}[(\tau^+_1)^{\kappa}]<\infty$, since $\kappa<k$. Thus, both conditions on the right-hand side of \eqref{eq-momentssubordinator} hold if and only if $\EE[|X_1|^{\kappa+1} ]<\infty$ which finishes the proof.
\end{proof}

 Finally, we consider the oscillating case of $\psi'(0+)=0$. Again, in the light of \eqref{eq_tau-_tau+} we need to investigate the existence of $\EE[(\tau_{Y_1}^+)^\kappa]$, where this time we restrict ourselves on finding conditions for $\EE[(\tau_{Y_1}^+)^\kappa]=\infty$. 

\begin{proof}[Proof of Theorem \ref{thm-existence} (iii)]
Assume that $\psi'(0+)= 0$. \\  
(a), $\kappa\in(0,1]$: From \eqref{eq_tau-_tau+} we have
$$	\mathbb{E}\left[(\tau_x^-)^\kappa\right] =\infty \qquad \text{if and only if} \qquad  \EE[(\tau_{Y_1}^+)^\kappa] =\infty,$$
and by Proposition \ref{lem-momentsubordination} for $\kappa \in (0,1)$ the latter follows in particular if  $\EE[Y_{1}^\kappa]=\infty$. 
This, however, is equivalent to $\int_{[1,\infty)} y^{\kappa} \Pi((y,\infty)) \diff y=\infty$ and via \eqref{eq-proof-partialintegration} it is furthermore equivalent to $\int_{[1,\infty)} y^{\kappa+1} \Pi(\diff y)= \infty$. \\
Consider now the case $\kappa = 1$, i.e. $\EE[Y_{1}]=\infty$. From \eqref{Lemma_derivative_inverse} it follows that $\Phi'(0+)=\mathbb{E}[\tau_1^+]=\infty$, and an inspection of the proof of Proposition \ref{lem-momentsubordination} reveals that in this setting also $\EE[\tau_{Y_1}^+] =\infty$. This again implies the statement.\\
(b) We consider a fixed $\kappa\in (\frac12,1)$ and prove \eqref{eq_kappa-thmoment_infinite} for the chosen $\kappa$. This will immediately imply the statement also for any $\kappa\geq 1$.\\
As before, from \eqref{eq_tau-_tau+} we have
		$$	\mathbb{E}\left[(\tau_x^-)^\kappa\right] =\infty \qquad \text{if and only if} \qquad  \EE[(\tau_{Y_1}^+)^\kappa] =\infty,$$
		where by Proposition \ref{lem-momentsubordination} the latter follows if  $\EE[(\tau_{1}^+)^\kappa]=\infty$.  
		Here, by \eqref{eq_Laplace_taux+}, \eqref{eq-Wolfemoment} and \eqref{eq_definition_marchaud_derivative}, 
	\begin{align}
		\mathbb{E}[(\tau_1^+)^\kappa]& = \left[ \mathbf{D}_q^\kappa e^{-\Phi(q)}\right]_{q=0} \nonumber \\
		&=  \left[\frac{\kappa}{\Gamma(1-\kappa)} \int_q^\infty \frac{e^{-\Phi(q)}- e^{-\Phi(u)}}{u^{\kappa+1}} \diff u  \right]_{q=0} \nonumber \\
		&=  \frac{\kappa}{\Gamma(1-\kappa)} \int_0^\infty \frac{1- e^{-\Phi(u)}}{u^{\kappa+1}} \diff u,\label{eq_proof_mainthm_iii2}
	\end{align}
where the left-hand side is finite if and only if the right-hand side is finite.\\
As  $\Phi$ is monotonically increasing with $\Phi(0)=0$ and  $\Phi(u)\overset{u\to\infty}{\longrightarrow}\infty$ we clearly have for all $\varepsilon>0$  
	\begin{equation*}
		\int_\varepsilon^\infty \frac{1-e^{-\Phi(u)}}{u^{\kappa+1}} \diff u \leq \int_\varepsilon^\infty \frac{1}{u^{\kappa+1}}\diff u <\infty.
	\end{equation*} 
Thus by \eqref{eq_proof_mainthm_iii2} , we have
	\begin{align} \label{eq_proof_mainthmiii3}
	\mathbb{E}[(\tau_1^+)^\kappa]  <\infty  \quad \text{ if and only if }\quad  \int_0^\varepsilon \frac{1-e^{-\Phi(u)}}{u^{\kappa+1}}<\infty \text{ for some }\varepsilon>0.
\end{align}
By Taylor's expansion, as $u\downarrow 0$, the term $1-e^{-\Phi(u)}$ is of the same order as $u \Phi'(u) e^{-\Phi(u)}$. Moreover, by \eqref{Lemma_derivative_inverse},
	\begin{equation*} 
	\lim_{u\downarrow 0} \frac{u \Phi'(u)}{u^{\kappa}} = \lim_{u\downarrow 0} \frac{\Phi'(u)}{u^{\kappa-1}} = \lim_{u\downarrow 0} \frac{u^{1-\kappa}}{\psi'(\Phi(u))}.
\end{equation*} 
Recall that $\kappa\in(\frac12 ,1)$ and $\psi''(0+)<\infty$. By a twofold application of l'Hospital's rule we get
\begin{equation*} 
	\begin{aligned}
		\lim_{u\downarrow 0} \frac{u^{1-\kappa}}{\psi'(\Phi(u))} &= \lim_{u\downarrow 0} \frac{(1-\kappa)\cdot u^{-\kappa}}{\psi''(\Phi(u))\cdot\Phi'(u)} 
		=\frac{(1-\kappa)}{\psi''(0+)} \cdot \lim_{u\downarrow 0} \frac{\psi'(\Phi(u))}{u^\kappa} \\
		&=\frac{(1-\kappa)}{\psi''(0+)} \cdot  \lim_{u\downarrow 0} \frac{\psi''(\Phi(u))\cdot\Phi'(u)}{\kappa \cdot u^{\kappa-1}}
		=\frac{(1-\kappa)}{\kappa} \cdot  \lim_{u\downarrow 0} \frac{u^{1-\kappa}}{\psi'(\Phi(u))}.
	\end{aligned} 
\end{equation*}
As $\kappa\neq \frac12$ this can only be true if
\begin{equation} \label{eq_Phi_behaviour2}
	\lim_{u\downarrow 0} \frac{u^{1-\kappa}}{\psi'(\Phi(u))} =  \lim_{u\downarrow 0} \frac{\psi'(\Phi(u))}{u^\kappa}  = \text{ either }0 \text{ or }\infty,
\end{equation}
which in turn implies  
\begin{equation*} 
	\lim_{u\downarrow 0} \frac{u\Phi'(u)}{u^\kappa} = \lim_{u\downarrow 0} \frac{u^{1-\kappa}}{\psi'(\Phi(u))} \cdot \frac{\psi'(\Phi(u))}{u^\kappa} = \lim_{u\downarrow 0} u^{1-2\kappa} = \infty. 
\end{equation*}
Thus also 
$$\lim_{u\downarrow 0} \frac{1-e^{-\Phi(u)}}{u^{\kappa}} = \lim_{u\downarrow 0} \frac{u\Phi'(u)e^{-\Phi(u)}}{u^{\kappa}} =\infty,$$
and in particular for any $C>0$ there exists $u_0>0$ such that $\frac{1-e^{-\Phi(u)}}{u^{\kappa}}>C$ for all $u<u_0$.
Hence 
	\begin{equation*}
	\int_0^\varepsilon \frac{1-e^{-\Phi(u)}}{u^{\kappa+1}} \diff u \geq \int_0^{u_0\wedge \varepsilon} \frac{1-e^{-\Phi(u)}}{u^{\kappa+1}} \diff u \geq \int_0^{u_0\wedge \varepsilon} \frac{C\cdot u^\kappa}{u^{\kappa+1}} = C\cdot\int_0^{u_0\wedge \varepsilon} \frac{1}{u} \diff u = \infty. 
\end{equation*}
By \eqref{eq_proof_mainthmiii3} this implies $\EE[(\tau_1^+)^\kappa]=\infty$ and thus the statement. 

Lastly, note that \eqref{eq_kappa-thmoment_infinite} for all $x>0$, $\kappa \geq 1$ is a direct consequence of (a) and (b), as either $\psi''(0+)<\infty$ in which case we can apply (b), or $\psi''(0+)=\infty$, which is equivalent to $\int_{[1,\infty)} y^2 \Pi(\diff y) = \infty$ and hence $\kappa^*=1$ is a possible choice in (a).
\end{proof}

\section*{Acknowledgements}
We would like to thank the reviewers for their helpful and constructive comments that helped us to improve this manuscript.

\small
\bibliography{literatureTheoryTTR}

\end{document}